\documentclass{amsart}
\usepackage{graphicx} 
\usepackage[a4paper,top=2.5cm,bottom=2.5cm,left=2.5cm,right=2.5cm]{geometry}

\usepackage{amsmath}
\usepackage{amsthm}
\usepackage{amssymb}
\usepackage[inline]{enumitem}

\usepackage{xcolor}

\usepackage{hyperref} 

\newtheorem{theorem}{Theorem}[section]

\newtheorem{lemma}[theorem]{Lemma}
\newtheorem{corollary}[theorem]{Corollary}
\newtheorem{question}[theorem]{Question}

\newtheorem*{remark}{Remark}

\title{On the number of components of twisted torus links}
\author{Adnan}
\address{Department of Mathematics, Kangwon National University, Chuncheon 24341, Republic of Korea}
\email{adnanshahab35@kangwon.ac.kr}

\author{Thiago de Paiva}
\address{Beijing International Center for Mathematical Research, Peking University, Beijing 100871, China P.R.}
\email{thhiagodepaiva@gmail.com}

\author{Kyungbae Park}
\address{Department of Mathematics, Kangwon National University, Chuncheon 24341, Republic of Korea}
\email{kyungbaepark@kangwon.ac.kr}



\begin{document}
\begin{abstract}
Twisted torus links $T(p,q;r,s)$ generalize torus links by introducing  $s$ additional twists on $r$ adjacent strands of the torus link $T(p,q)$. It is well known that the number of components of a torus link $T(p, q)$ is given by the greatest common divisor of $p$ and $q$. However, the number of components of a twisted torus link cannot be determined as straightforwardly from its parameters. In this work, we present a Euclidean algorithm-like procedure for computing the number of components of twisted torus links from their parameters. As a result, we show that the number of components of a twisted torus link $T(p, q; r, s)$ is a multiple of $\gcd(p, q, r, s)$, and in particular, $T(p, q; r, s)$ is a knot only if $\gcd(p, q, r, s) = 1$. We also use our algorithm to determine the number of components for several families of twisted torus links considered in the literature.
\end{abstract}
\maketitle

\section{Introduction}
Given positive integers $p\geq r>0$, and integers $q$ and $s$, a \emph{(generalized) twisted torus link}\footnote{In \cite{Paiva-2023-2} the term \emph{twisted torus link} refers only to $T(p,q;r,s)$ when $s$ is a multiple of $r$. In this paper we consider more general cases.} $T(p,q;r,s)$ is a generalization of torus links obtained by introducing $s$ additional twists on the first $r$ adjacent strands of the standard braid representation of the $(p, q)$-torus link. It can also be precisely described as the closure of a braid with $p$ strands of the following braid word:
\[
    (\sigma_1\sigma_2\cdots\sigma_{p-1})^q(\sigma_1\cdots\sigma_{r-1})^{s}
\]
See Figure \ref{Fig:TTL} for a diagram of the twisted torus link $T(9,6;7,4)$. 

In particular, if $p$ and $q$ are relatively prime and $s$ is a multiple of $r$, then the twisted torus link $T(p,q;r,s)$ is in fact a knot, known as a \emph{twisted torus knot}. Twisted torus knots form a well-studied family of knots with applications in low-dimensional topology. They were introduced by Dean to study Seifert fibered spaces obtained via Dehn fillings \cite{Thesis}. Additionally, they have been used to provide important examples in the study of simple hyperbolic knots (in terms of the number of ideal tetrahedra in the exterior) \cite{Callahan-Dean-Weeks-1999, Champanerkar-Kofman-Patterson-2004} and Heegaard splittings \cite{Moriah-Sedgwick-2009}. Their properties and invariants have been extensively studied. For example, their geometric type has been explored in \cite{lee2018satellite, LeeTorusknotsobtained, unexpected, LeeThiago}, their bridge spectra in \cite{Bridge}, their Alexander polynomial in \cite{Morton-2006,Adnan-Park-2024}, their knot Floer homology in \cite{Vafaee-2015}, and their Jones polynomial in \cite{Bavier-Doleshal-2023}. 

Twisted torus links can also be compared to $T$-links, which were introduced by Birman and Kofman in \cite{Birman_09} to describe Lorenz links. More precisely, twisted torus links $T(p, q; r, s)$ form a special subclass of $T$-links when $q, s>0$; see Section~\ref{sec:$T$-links}. Moreover, twisted torus links are part of the broader family of generalized $T$-links, a class of links that encompasses all links in the 3-sphere; see \cite{Volumebounds}.

One of the fundamental questions in link theory is determining when a link is a knot, or more generally, finding the number of its components. For torus links $T(p,q)$, the number of components is given by $\gcd(p,q)$. However, determining the number of components of a twisted torus link $T(p,q;r,s)$ is not straightforward and cannot be easily deduced from its parameters. Note that the condition for twisted torus knots, namely that $\gcd(p,q) = 1$ and $s$ is a multiple of $r$, is a sufficient but not necessary condition for $T(p,q;r,s)$ to be a knot. For example, one can verify that $T(5,4;3,2)$ is a knot. More generally, $T(2n+3, 2n+2; 2n+1, 2n)$ is a knot for every $n\geq1$.

Let \(NC(p,q;r,s)\) denote the number of components of the twisted torus link \(T(p,q;r,s)\). In \cite[Theorem 2.1]{LC-2016}, a condition on the parameters \(p\), \(q\), \(r\), and \(s\) is given that determines the parity of \(NC(p,q;r,s)\), that is, whether \(NC(p,q;r,s)\) is even or odd. We also note that, in \cite{DHMZ-2018}, the general braid is used to determine the number of components of certain twisted torus links and to identify the link type of several examples. The main result of this paper is to present a Euclidean algorithm-like arithmetic procedure for computing the number of components of twisted torus links from their parameters.

\begin{theorem}\label{thm:main}
For positive integers $p\geq r> 0$ and $q,s\in\mathbb{Z}$, let $NC(p,q;r,s)$ denote the number of components of the twisted torus link $T(p,q;r,s)$, and $[x]_m$ denote the residue of $x$ modulo $m$. Define 
\[(p_1, q_1, r_1, s_1) = (p, [q]_p, r, [s]_r),\] 
and let $\{(p_i, q_i, r_i, s_i)\}_{i=1}^n$ be the sequence of quadruples obtained by the following recursive procedure:\\
\indent If $q_{i} = 0$ or $s_{i} = 0$, let $n=i$. Otherwise,
    \[(p_{i+1}, q_{i+1}, r_{i+1}, s_{i+1}):=\begin{cases}
        (q_{i}, [p_{i}]_{q_{i}}, r_{i}, [-s_{i}]_{r_{i}}), & \text{if } q_{i}\geq r_{i}\\
        (r_{i}, [s_{i}+q_{i}]_{r_{i}}, q_{i}, [r_{i}-p_{i}]_{q_{i}}) & \text{if } q_{i}<r_{i}.
    \end{cases}\]
Then $\{(p_i, q_i, r_i, s_i)\}_{i=1}^{n}$ is a finite sequence, $NC(p,q;r,s)=NC(p_i, q_i; r_i, s_i)$ for each $i=1,\dots,n$, and
\[NC(p,q;r,s)=\begin{cases}
    p_n - r_n + \gcd(r_n, s_n)  & \text{if } q_n = 0\\
    \gcd(p_n, q_n)              & \text{if } s_n = 0\\
\end{cases}\]
\end{theorem}

\begin{figure}[t!]
    \centering
    \includegraphics[width=0.3\textwidth,angle=90]{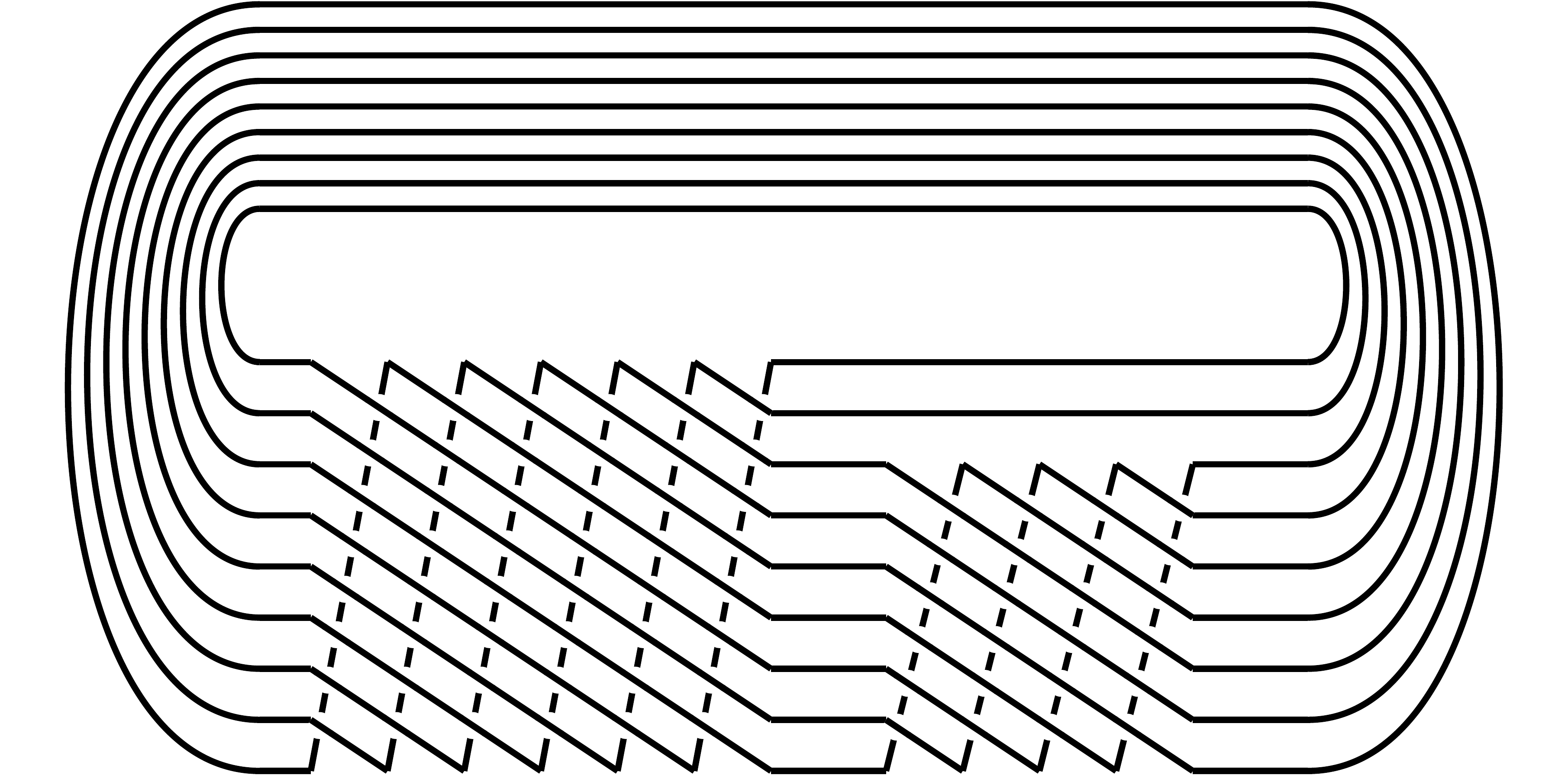}
    \caption{A diagram of the twisted torus link $T(9,6;7,4)$.}
    \label{Fig:TTL}
\end{figure}

\begin{remark}
    If $r=0$, then the twisted torus link $T(p,q;r,s)$ reduces to the torus link $T(p,q)$. In this case, the parameters $r$ and $s$ can be disregarded,  and the recursive structure of the algorithm simplifies to the Euclidean algorithm for computing $\gcd(p,q)$. Therefore, this special case recovers the well-known result that the number of components of the torus link $T(p,q)$ is $\gcd(p,q)$.
\end{remark}

Our proof of this result relies on observing operations on twisted torus links that preserve the number of components while reducing the number of strands. Note that the quadruples in the sequence $\{(p_i, q_i, r_i, s_i)\}_{i=1}^n$ from Theorem \ref{thm:main} correspond to a sequence of twisted torus links $T(p_i,q_i;r_i,s_i)$, all of which have the same number of components. The twisted torus link corresponding to $(p_n, q_n, r_n, s_n)$ is either a torus link or the union of a torus link and several unknotted components.

Although Theorem \ref{thm:main} does not provide an explicit formula for $NC(p,q;r,s)$ in terms of its parameters, it determines $NC(p,q;r,s)$ for the families considered in \cite{LC-2016}, which were previously studied through explicit computations using a computer. This will be discussed in Section \ref{sec:conjectures}. Another consequence of our proposed algorithm is that it provides information about the number of components of $T(p,q;r,s)$ in terms of the parameters.

\begin{theorem}\label{thm:gcd}
    The number of components of a twisted torus link $T(p,q;r,s)$ is a positive multiple of $\gcd(p,q,r,s)$. 
\end{theorem}
\begin{proof}
    Observe that $\gcd(p_i,q_i,r_i,s_i)$ divides $\gcd(p_{i+1},q_{i+1},r_{i+1},s_{i+1})$ in the sequence in Theorem \ref{thm:main}. Moreover, $\gcd(p_n,q_n,r_n,s_n)$ divides $NC(p,q;r,s)$ in Theorem \ref{thm:main}.
\end{proof}

In particular, we have the following necessary condition for a twisted torus link to be a knot.
\begin{corollary}
    If $T(p,q;r,s)$ is a knot, then $\gcd(p,q,r,s)=1$.
\end{corollary}

Note that the converse of the above corollary is not true. For example, twisted torus links $T(p,q;r,r)$ with $\gcd(p,q)=n$ and $\gcd(q,r)=1$ (hence $\gcd(p,q,r,s)=1$) are $n$ component links. Therefore, the condition that $\gcd(p,q,r,s)=1$ is not sufficient to ensure that $T(p,q;r,s)$ is a knot. As mentioned above, the conditions defining \emph{twisted torus knots} provide a sufficient but not necessary criterion. Therefore, we pose the following question.

\begin{question} 
    Find a nonrecursive arithmetic criterion in the original parameters $p,q,r,s$ that determines when a twisted torus link $T(p,q;r,s)$ is a knot.
\end{question}

Of course, the above is a special case of the following more general question.
\begin{question}
    Is there a closed-form formula in terms of the parameters for determining the number of components of twisted torus links?
\end{question}

By viewing a twisted torus link as the closure of a braid, one can relate its
number of components to the associated permutation on the strands. Indeed, the number of components of the closure of a
braid is equal to the number of cycles in the disjoint cycle decomposition of the permutation induced by the braid; see, for example, \cite{Morton-Hadji-2005} and also \cite{DHMZ-2018} in a closely related twisted torus link setting. In our case, the number of components of \(T(p,q;r,s)\) is exactly
the number of cycles in the disjoint cycle decomposition of the permutation
\(\tau\circ \sigma\). Here \(\sigma\) and \(\tau\) are permutations on
\(\mathbb{Z}/p\mathbb{Z}=\{0,1,\dots,p-1\}\) defined by
\[
\sigma(i)=[i-q]_p,
\]
and
\[
\tau(i)=
\begin{cases}
[i-s]_r, & \text{if } i<r,\\
i, & \text{if } i\ge r.
\end{cases}
\]
Thus, our algorithm may also be regarded as a method for computing the number
of cycles of these permutations. Nevertheless, the proof itself relies on topological arguments.

\subsection*{Organization of the paper}
In Section \ref{sec:moves}, we examine certain transformations on twisted torus links that preserve the number of components, which form the foundation of our algorithm. In Section \ref{sec:conjectures}, we apply this algorithm to compute the number of components for various families of twisted torus links, including those considered in Section~4 of \cite{LC-2016}. In the final section, we discuss a generalization of our algorithm for determining the number of components of $T$-links with three pairs of parameters. Throughout this paper let $[x]_p$ denote the residue of $x$ modulo $p$.

\subsection*{Acknowledgments} During his visit to IMPA, Thiago de Paiva was partially supported by IMPA and by a grant from Professor Vinicius Ramos, funded by the Serrapilheira Institute. He also thanks Professor Mikhail Belolipetsky for his kind hospitality during his research stay at IMPA. Adnan and Kyungbae Park were supported by the National Research Foundation of Korea (NRF) grant funded by the Korea government (MSIT/MOE) (Nos. RS-2025-24523511 and RS-2025-25415913).

\section{Moves preserving the number of components}\label{sec:moves}
Our central observation in proving Theorem~\ref{thm:main} is that certain operations on twisted torus links preserve the number of components. For positive integers $p \geq r > 0$ and $q,s\in\mathbb{Z}$, let $NC(p,q;r,s)$ denote the number of components of the twisted torus link $T(p,q;r,s)$. As a preliminary remark, we note the following basic fact.
\begin{lemma}\label{lem:full_twists}
    Let $p\geq r> 0$ and $q, q',s, s'\in\mathbb{Z}$. If $q\equiv q'\pmod{p}$ and $s\equiv s'\pmod{r}$, then
    \[
        NC(p,q;r,s)=NC(p,q';r,s').
    \]
\end{lemma}
\begin{proof}
   Any additional full twists to a subset of strands in a closed braid do not alter the number of components of the link.
\end{proof}
The following lemma shows that, in the construction of a twisted torus link, the placement of the $s$ additional twists on any set of $r$ consecutive strands yields a link isotopic to the twisted torus link $T(p,q;r,s)$.
\begin{lemma}\label{lem:shift}
    For each $1\leq k\leq p-r+1$, the link given by the closure of the following $p$-strand braid
    \[
        (\sigma_1 \cdots \sigma_{p-1})^q(\sigma_{k}\cdots\sigma_{k+r-2})^s
    \]
    is isotopic to the twisted torus link $T(p,q;r,s)$.
\end{lemma}
\begin{proof}
    By the isotopy illustrated in Figure \ref{fig:shift}, the closure of the braid
    \[
        (\sigma_1 \cdots \sigma_{p-1})^q(\sigma_{k}\cdots\sigma_{k+r-2})^s
    \]
    for $1< k \leq p-r+1$ is isotopic to the closure of the braid
    \[
        (\sigma_1 \cdots \sigma_{p-1})^q(\sigma_{k-1}\cdots\sigma_{k+r-3})^s.
    \]
    Applying this isotopy inductively moves the twist region to start at position $k=1$, establishing the desired isotopy to the standard form of the twisted torus link $T(p,q;r,s)$. 
\end{proof}
\begin{figure}[t]
    \includegraphics[width=0.8\textwidth]{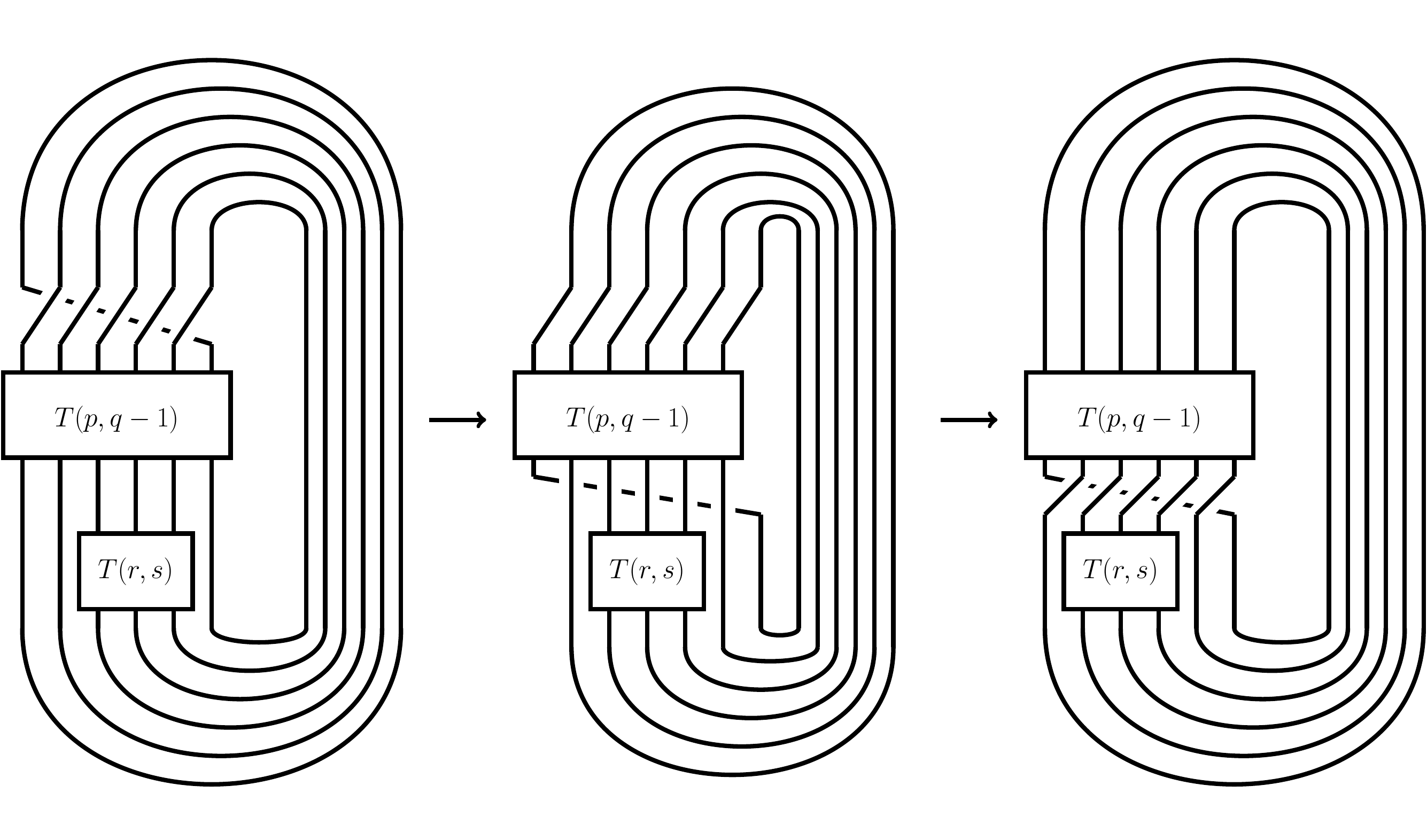}
    \caption{An isotopy that horizontally shifts the additional twist region.}
    \label{fig:shift}
\end{figure}

The following lemma considers the mirror of a twisted torus link.
\begin{lemma}\label{lem:mirror} 
    The mirror image of $T(p,q;r,s)$ is isotopic to $T(p,-q;r,-s)$. In particular, we have the identity:
    \[
        NC(p,q;r,s)=NC(p,-q;r,-s)
    \]
\end{lemma}
\begin{proof}       
    Consider the mirror image of $T(p,q;r,s)$, which is represented by the closure of the braid
    \[
        (\sigma_1^{-1}\cdots\sigma_{p-1}^{-1})^q(\sigma_1^{-1} \cdots \sigma_{r-1}^{-1})^s.
    \]
    Applying a horizontal reflection (braid flip) yields the braid
    \begin{align*}
        &(\sigma_{p-1}^{-1}\cdots\sigma_{1}^{-1})^q(\sigma_{p-1}^{-1} \cdots \sigma_{p-r+1}^{-1})^s\\
        =&(\sigma_{1}\cdots\sigma_{p-1})^{-q}(\sigma_{p-r+1} \cdots \sigma_{p-1})^{-s}
    \end{align*}
    By Lemma~\ref{lem:shift}, the closure of this braid is isotopic to $T(p,-q;r,-s)$.  
\end{proof}

The following lemma computes the number of components of the twisted torus link in the final step of the algorithm described in Theorem~\ref{thm:main}.
\begin{lemma}\label{lem:final_step} 
Let $p\geq r> 0$ and $q,s\in\mathbb{Z}$. Then we have the following:
    \begin{enumerate}
        \item $NC(p,q;r,s)= (p-r) + \gcd(r,s)$ if $[q]_p=0$
        \item $NC(p,q;r,s)=\gcd(p,q)$ if $[s]_r= 0$
    \end{enumerate} 
\end{lemma}
\begin{proof}        
    If $[q]_p=0$, then by Lemma \ref{lem:full_twists}, we have $NC(p,q;r,s)=NC(p,0;r,s)$. The link $T(p,0;r,s)$ consists of the torus link $T(r,s)$ along with $p-r$ unknots, resulting in a total of $(p-r) + \gcd(r,s)$ components. 
    
    Similarly, if $[s]_r=0$, then $NC(p,q;r,s)=NC(p,q;r,0)$ by Lemma \ref{lem:full_twists}. In this case, the link $T(p,q;r,0)$ reduces to the torus link $T(p,q)$, which has $\gcd(p,q)$ components.
\end{proof}

The following two key lemmas ensure that our algorithm described in Theorem~\ref{thm:main} functions correctly, as they relate a given twisted torus link to another with fewer strands while preserving the number of components. We begin with the case where $p \geq q \geq r$.
\begin{lemma}\label{lem:swapping_p_and_q}
    Consider the twisted torus link $T(p,q;r,s)$ with $p \geq q \geq r>0$, and $s\in\mathbb{Z}$. Then we have the identities:
    \[
        NC(p,q;r,s)=NC(q,p;r,-s)=NC(q,-p;r,s).
    \]
\end{lemma}
\begin{proof}
    Place the twisted torus link $T(p,q;r,s)$ on a (thickened) flat torus as illustrated in the first diagram of Figure~\ref{fig:swapping_p_and_q}. Apply an isotopy that transforms the $(p,q)$ torus braid into the $(q,p)$ torus braid; this corresponds to a 180-degree rotation about a diagonal of the flat torus. After this rotation, the sub-braid $B=(\sigma_1\sigma_2\cdots\sigma_{r-1})^{s}$ is transformed into $B'=(\sigma_{r-1}\sigma_{r-2}\cdots\sigma_1)^s$, embedded within the braid $(\sigma_1\sigma_2\dots\sigma_{q-1})^p$. Since $B'$ involves at most $q$ strands, it can be pushed downward to lie vertically after the $(q,p)$ torus braid, as shown in Figure~\ref{fig:swapping_p_and_q}. This yields the braid
    \[
        (\sigma_1\dots\sigma_{q-1})^p(\sigma_{q-1}\dots\sigma_{q-r+1})^{s},
    \]
    whose closure has the same number of components as the closure of the braid:
    \begin{align*}
        &(\sigma_1\cdots\sigma_{q-1})^p(\sigma_{q-1}^{-1}\cdots\sigma_{q-r+1}^{-1})^{s}\\
        =&(\sigma_1\cdots\sigma_{q-1})^p(\sigma_{q-r+1}\cdots\sigma_{q-1})^{-s}.
    \end{align*}
    By Lemma~\ref{lem:shift}, this link is isotopic to the closure of the braid
    \[
        (\sigma_1\sigma_2\cdots\sigma_{q-1})^p(\sigma_{1}\sigma_{2}\cdots\sigma_{r-1})^{-s},
    \]
    which represents the twisted torus link $T(q,p;r,-s)$.
    The final identity follows from Lemma~\ref{lem:mirror}, since $T(q,-p;r,s)$ is the mirror image of $T(q,p;r,-s)$.
\end{proof}
\begin{figure}[t]
    \centering
    \includegraphics[width=1\textwidth]{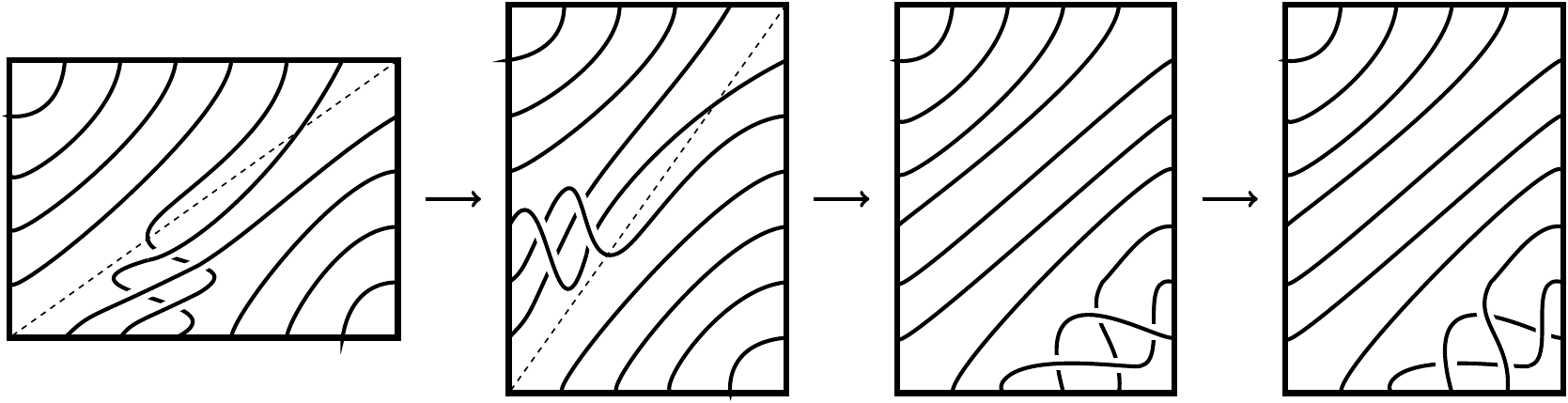}
    \caption{A transformation of $T(6,4;3,2)$ to $T(4,6;3,-2)$ on a thickened flat torus.}
    \label{fig:swapping_p_and_q}
\end{figure}

Now, we turn to the case where $p\geq r>q$.
\begin{lemma}\label{lem:forming_r_strands}
    Consider the twisted torus link $T(p,q;r,s)$ with $p \geq r > q > 0$, and $s>0$. Then, \[NC(p,q;r,s)=NC(r,s+q;q,r-p).\]  
\end{lemma}
\begin{proof}
    If $p=r$, the two powers of $(\sigma_1\cdots\sigma_{p-1})$ combine, so $NC(p,q;p,s)=\gcd(p,q+s)=NC(p,s+q;q,0)$, as desired. Now assume $p>r>q>0$ and $s>0$. By \cite[Proposition 3.3]{Lorenzknots}, the twisted torus link $T(p,q;r,s)$ is isotopic to the link represented by the braid with $r$ strands
    \[
        (\sigma_1\sigma_2\dots\sigma_{r-1})^{q}(\sigma_{r-1}\sigma_{r-2}\dots\sigma_{r-q+1})^{p-r}(\sigma_1\sigma_2\dots\sigma_{r-1})^{s},
    \]
    as illustrated in the second diagram of Figure~\ref{fig:forming_r_strands} for the case of the twisted torus link $T(8, 4; 5, 3)$.

    We reposition the last $s$ horizontal strands of the sub-braid $(\sigma_1\sigma_2\dots\sigma_{r-1})^{s}$ around the braid closure to obtain the braid
    \[
        (\sigma_1\sigma_2\dots\sigma_{r-1})^{s+q}(\sigma_{q-1}\sigma_{q-2}\dots\sigma_{1})^{p-r},
    \]
    as illustrated in the fourth diagram of Figure~\ref{fig:forming_r_strands}.

    Next, we change all crossings in the sub-braid $(\sigma_{q-1}\sigma_{q-2}\dots\sigma_{1})^{p-r}$. While this operation alters the linking type, it preserves the number of components and yields the braid 
    \begin{align*}
        &(\sigma_1\sigma_2\dots\sigma_{r-1})^{s+q}(\sigma_{q-1}^{-1}\sigma_{q-2}^{-1}\dots\sigma_{1}^{-1})^{p-r}\\
        =&(\sigma_1\sigma_2\dots\sigma_{r-1})^{s+q}(\sigma_{1}\sigma_{2}\dots\sigma_{q-1})^{-(p-r)}, 
    \end{align*}
    which represents the twisted torus link $T(r,s+q;q,r-p)$, as desired.
\end{proof}

\begin{figure}[t]
    \centering
    \includegraphics[width=\textwidth]{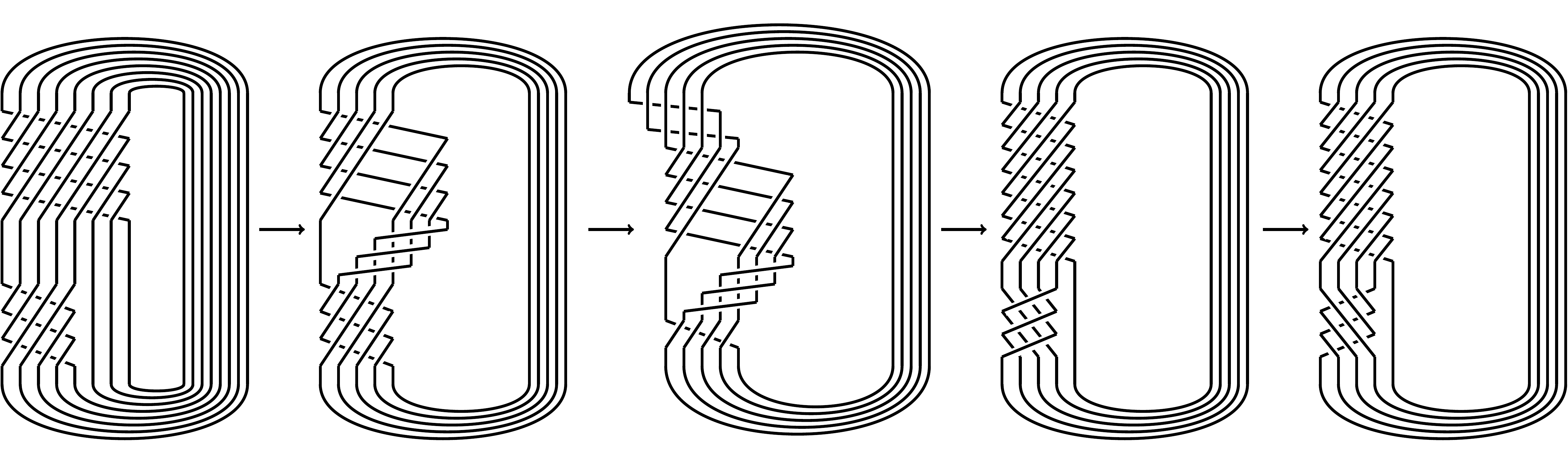}
    \caption{A transformation from $T(8,4;5,3)$ to $T(5,7;4,-3)$ that preserves the number of components.}
    \label{fig:forming_r_strands}
\end{figure}

We are now ready to prove Theorem~\ref{thm:main}.
\begin{proof}[Proof of Theorem \ref{thm:main}]
    Let $T(p,q;r,s)$ be a twisted torus link with positive integers $p\geq r> 0$ and $q,s\in\mathbb{Z}$, and let $\{(p_i,q_i,r_i,s_i)\}_{i=1}^n$ be the sequence of quadruples defined according to the algorithm described in Theorem~\ref{thm:main}.
    
    We first show that the sequence $\{(p_i,q_i,r_i,s_i)\}_{i=1}^n$ terminates. Observe that, by the recursive formula and an induction argument, we have $p_i>q_i\geq0$ and $p_i\geq r_i>s_i\geq 0$ for all $i$. If $p_i=r_i$ for some $i$, then since $q_i<r_i$, it follows that $s_{i+1}=0$, and hence the sequence terminates at this step. Now, suppose $p_i>r_i$ for all $i$. Since $p_{i+1}$ is either $q_i$ or $r_i$, and both are strictly less than $p_i$, it follows that $\{p_i\}$ is strictly decreasing. Therefore, the sequence must eventually terminate when either $q_{i} = 0$ or $s_{i} = 0$.
    
    It follows directly from Lemmas~\ref{lem:full_twists}, \ref{lem:swapping_p_and_q}, and \ref{lem:forming_r_strands} that 
    \[
        NC(p,q;r,s)=NC(p_{1}, q_{1}; r_{1}, s_{1})
    \] 
    and 
    \[
        NC(p_{i}, q_{i}; r_{i}, s_{i})=NC(p_{i+1}, q_{i+1}; r_{i+1}, s_{i+1})
    \] for all $i=1,\dots ,n-1$. Therefore, we conclude that $NC(p,q;r,s)=NC(p_n,q_n;r_n,s_n)$, where the final value is computed using Lemma~\ref{lem:final_step}.
\end{proof}

\section{Examples}\label{sec:conjectures}
In this section, we apply our algorithm to specific families of twisted torus links to illustrate how it effectively determines their number of components.

\subsection{Case: $r=2$}
One of the special cases of twisted torus knots or links is when $r=2$. In particular, for twisted torus knots $T(p,q;2,2s')$ with $\gcd(p,q)=1$, the Alexander polynomial is computed in \cite{Morton-2006}, Knot Floer homology is discussed in \cite{Vafaee-2015}, and the Jones polynomial is studied in \cite{Bavier-Doleshal-2023}. 
\begin{theorem}
    For the twisted torus link $T(p,q;2,s)$ with $p\geq 2$ and $q,s\in\mathbb{Z}$, we have
    \[
        NC(p,q;2,s)=\begin{cases}
            \gcd(p,q) &\text{if $s$ is even},\\
            2 &\text{if $\gcd(p,q)=1$ and $s$ is odd},\\ 
            \gcd(p,q)-1 & \text{if $\gcd(p,q)\geq 2$ and $s$ is odd }.\\
        \end{cases}
    \]
\end{theorem}
\begin{proof}
    This result follows directly from Theorem~\ref{thm:main}. Let $\{(p_i,q_i,r_i,s_i)\}$ be the sequence defined therein. If $s$ is even, then $(p_1,q_1,r_1,s_1)=(p,[q]_p,2,0)$, and the sequence terminates immediately with $n=1$. Hence, $NC(p,q;2,s)=\gcd(p,q)$.
    
    Now suppose $s$ is odd and $\gcd(p,q) = 1$. Then $(p_1,q_1,r_1,s_1)=(p,[q]_p,2,1)$. As long as $q_i\geq r$ for all $i\leq k$, the sequence $\{(p_i,q_i)\}_{i=1}^k$ mirrors the steps of the Euclidean algorithm for computing $\gcd(p,q)$. Since $\gcd(p,q)=1$, this leads to the sequence $\{(p_n,q_n,r_n,s_n)\}$: 
    \[
        (p,[q]_p,2,1), \dots, (p_{n-1},1,2,1),(2,0,1,0),
    \]
    where $p_{n-1}\geq 2$, and thus $NC(p,q;2,s)=2$.

   The final case, where $s$ is odd and $\gcd(p,q)\geq 2$, will be treated in the following theorem, which addresses a more general setting.
\end{proof}

\subsection{Case: $\gcd(p,q)\geq r$} 
A special case in which an explicit formula for $NC(p,q;r,s)$ can be deduced in terms of the parameters is when $\gcd(p,q)\geq r$.
\begin{theorem}
    For the twisted torus link $T(p,q;r,s)$ with $\gcd(p,q)\geq r$, we have 
    \[
        NC(p,q;r,s)=\gcd(p,q)-r+\gcd(r,s).
    \] 
\end{theorem}
\begin{proof}
    If $r\mid s$, Lemma~\ref{lem:final_step}(2) gives $NC(p,q;r,s)=\gcd(p,q)$, and the stated formula follows from $\gcd(r,s)=r$. We may therefore assume $r\nmid s$.

    Since $\gcd(p,q)\geq r$, the pairs $(p_i,q_i)$ follow the steps of the Euclidean algorithm applied to $p$ and $q$, while $r_i=r$ and $\gcd(r_i,s_i)=\gcd(r,s)$ throughout. The sequence therefore terminates at $(\gcd(p,q),0,r,s_n)$, and the result follows from Theorem~\ref{thm:main}.
\end{proof}

\subsection{Case: $s=\pm q$} 
As mentioned earlier, several conjectures regarding the number of components of  specific families of twisted torus links, particularly those with \texorpdfstring{$s=\pm q$}{s=±q}, were proposed in \cite{LC-2016}, based on computational evidence. In this section, we apply our algorithm to determine the number of components for the families considered in Section~4 of \cite{LC-2016}.

The following result addresses twisted torus links $T(p,q;r,\pm q)$ under specific conditions on the parameter $r$.
\begin{theorem}\label{theo4}
    For the twisted torus link $T(p,q;r,\pm q)$ with $p\geq r>0$, $q>0$, and $p \not\equiv 0\pmod{q}$, the number of components is given as follows:\\
    \begin{itemize}
        \item If $r\equiv 1$, $q$, or $2q-1 \pmod{2q}$, then
        \[
            NC(p,q;r,q)=\gcd(p,q).
        \] 
        \item If $r>q$, then
        \[
            NC(p,q;r,-q)= r-q+\gcd(q,r-p).
        \]
    \end{itemize}
\end{theorem}
\begin{proof} 
    Write $a=[r-p]_q$. First, suppose $r \equiv 1 \pmod{2q}$. If $r=1$, then clearly $NC(p,q;1,q)=\gcd(p,q)$. Assume $r>2q$. The first step of Theorem~\ref{thm:main} gives
    \[
        (p,q,r,q)\longrightarrow(r,2q,q,a).
    \]
    If $a=0$, the sequence stops here, and $NC(p,q;r,q)=\gcd(r,2q)=1=\gcd(p,q)$. If $a>0$, the next steps are
    \[
        (r,2q,q,a)\longrightarrow(2q,1,q,[-a]_q)
        \longrightarrow(q,[1-a]_q,1,0).
    \]
    Thus $NC(p,q;r,q)=\gcd(q,1-a)=\gcd(p,q)$, since $a\equiv r-p\equiv 1-p\pmod q$.

    Now suppose $ r \equiv q \pmod{2q}$. If $r = q$, then $(p_1, q_1, r_1, s_1) = (p, q, q, 0)$ and $NC(p,q;r,q) = \gcd(p,q) $. For $r > q$, the sequence becomes:
    \[
        (p, q, r, q),\ (r, 2q, q, [r - p]_q),\ (2q, q, q, [p - r]_q),\ (q, 0, q, [r - p]_q),
    \]
    which gives $NC(p,q;r,q) = \gcd(q, [r - p]_q) = \gcd(p,q)$, since $r \equiv q \pmod{2q}$.

    Finally, suppose $r \equiv 2q - 1 \pmod{2q}$. The first step gives
    \[
        (p,q,r,q)\longrightarrow(r,[2q]_r,q,a).
    \]
    If $a=0$, the sequence stops here, and $NC(p,q;r,q)=\gcd(r,[2q]_r)=1=\gcd(p,q)$. If $a>0$ and $r=2q-1$, the remaining step is
    \[
        (r,1,q,a)\longrightarrow(q,[1+a]_q,1,0).
    \]
    If $a>0$ and $r>2q-1$, the remaining steps are
    \[
        (r,2q,q,a)\longrightarrow(2q,2q-1,q,[-a]_q)
        \longrightarrow(2q-1,1,q,a)
        \longrightarrow(q,[1+a]_q,1,0).
    \]
    In both cases with $a>0$, $NC(p,q;r,q)=\gcd(q,1+a)=\gcd(p,q)$, since $a\equiv r-p\equiv-1-p\pmod q$.

    Finally, consider the case $s = -q$ and $r > q$. Then the sequence becomes:
    \[
        (p, q, r, [-q]_r),\ (r, 0, q, [r - p]_q),
    \]
    which yields $NC(p,q;r,-q) = r - q + \gcd(q, [r - p]_q)$.
\end{proof}

We now consider twisted torus links $T(p,q;r,\pm q)$ under specific conditions on the parameter $p$. Since these results follow directly from Theorem~\ref{thm:main}, we omit the proofs and leave them to the reader.
\begin{theorem}
    For the twisted torus link $T(p,q;r,\pm q)$ with $[p]_q=1$ and $q>0$, the number of components is given as follows:
    \begin{itemize}
    \item If $1\leq [r]_{2q}\leq q$, then $NC(p,q;r,q)= \gcd([r]_{2q},q-1)$.
    \item If  $q+1\leq [r]_{2q}\leq 2q-1$, then $NC(p,q;r,q)= \gcd([r]_{2q}+2,q+1)$.
    \item If $1< r \leq q$, then $NC(p,q;r,-q)= \gcd(r,q+1)$.
    \end{itemize}
\end{theorem}        

\begin{theorem}          
    For the twisted torus link $T(p,q;r,\pm q)$ with $p\equiv -1\pmod{q}$ and $q>0$, the number of components is given as follows:
    \begin{itemize}
    \item If $1\leq [r]_{2q}\leq q$, then $NC(p,q;r,q)= \gcd([r]_{2q},q+1)$.      
    \item If $q+1\leq [r]_{2q}\leq 2q-1$, then $NC(p,q;r,q)= \gcd([r]_{2q}-2,q-1)$.      
    \item If $r<q$, then $NC(p,q;r,-q)=\gcd(r,q-1)$.
    \end{itemize}
\end{theorem}
    
\begin{theorem}       
    For the twisted torus link $T(p,q;r,\pm q)$ with $q>0$, the number of components is given as follows:
    \begin{itemize}
    \item If $[r]_{2q}=k$ and $[p]_q=0$, then
    \[
        NC(p,q;r,q)= |q-k|+\gcd(k,q) \quad\text{and}\quad NC(p,q;r,-q)= |r-q|+\gcd(k,q).
    \]
    \item If $[r]_{2q}=0$ and $[p]_q\neq 0$, then 
    \[
        NC(p,q;r,q) = q+\gcd(q,p)\quad\text{and}\quad NC(p,q;r,-q) = r-q+\gcd(q,p).
    \]          
    \item If $[p]_q=[r]_{2q}$, then
    \[
        NC(p,q;r,q)=\gcd(r,2q)\quad\text{and}\quad NC(p,q;r,-q)=r.
    \]
    \end{itemize}
\end{theorem}

\section{$T$-links with three pairs of parameters}\label{sec:$T$-links}
Twisted torus links can be compared with $T$-links, a family of links introduced by Birman and Kofman \cite{Birman_09}. In their convention, given positive integers $2\leq p_1\leq p_2\leq\dots\leq p_n$ and $q_i>0$ for $i=1,\dots,n$, a $T$-link $T(p_1,q_1;p_2,q_2;\dots;p_n,q_n)$ is defined as the closure of the braid 
\[
    (\sigma_1\sigma_2\cdots\sigma_{p_1-1})^{q_1}(\sigma_1\sigma_2\cdots\sigma_{p_2-1})^{q_2}\dots(\sigma_1\sigma_2\cdots\sigma_{p_n-1})^{q_n}
\]
on $p_n$ strands. This naturally leads to the following fundamental question:
\begin{question}
    What is the number of components of a $T$-link?
\end{question}
In particular, the subclass of $T$-links with $n=2$ is contained within the broader class of twisted torus links. Our results provide an answer for this case. We now introduce an algorithm that enables the computation of the number of components of a $T$-link when $n=3$, which serves as a natural generalization of the $n=2$ case. 

From this point forward, we generalize the notion of a $T$-link by allowing arbitrary orderings of the $p_i$ and arbitrary integer exponents $q_i$. Specifically, for integers $p_i>0$ and $q_i\in\mathbb{Z}$ for $i=1,2,3$, we define $T(p_1,q_1;p_2,q_2;p_3,q_3)$ as the closure of the braid
\[
    (\sigma_1\sigma_2\dots\sigma_{p_1-1})^{q_1}(\sigma_1\dots\sigma_{p_2-1})^{q_2}(\sigma_1\dots\sigma_{p_3-1})^{q_3}
\]
on $\max\{p_1,p_2,p_3\}$ strands. The number of components of the link is denoted by $NC(p_1,q_1;p_2,q_2;p_3,q_3)$. We say that $T(p_1,q_1;p_2,q_2;p_3,q_3)$ is in \emph{standard form} if $p_1\geq p_2\geq p_3$ and $p_i> q_i\geq 0$ for each $i=1,2,3$.

The following lemma is analogous to Lemma~\ref{lem:full_twists}. 
\begin{lemma}\label{lem:full_twists-2}
    $NC(p_1,q_1;p_2,q_2;p_3,q_3)=NC(p_1,[q_1]_{p_1};p_2,[q_2]_{p_2};p_3,[q_3]_{p_3})$
\end{lemma}
\begin{proof}
    Adding or subtracting full twists on $p_i$ strands does not change the number of components.  
\end{proof}

The following result shows that any generalized $T$-link with $n=3$ can be transformed into a $T$-link in standard form without altering the number of components. 
\begin{lemma}\label{lem:standard_form}
    For any $T(p_1,q_1;p_2,q_2;p_3,q_3)$, there exists a $T$-link $T(p_1',q_1';p_2',q_2';p_3',q_3')$ in standard form such that $(p_1',p_2',p_3')$ is a permutation of $(p_1,p_2,p_3)$, and the two links have the same number of components.
\end{lemma}
\begin{proof}
    By applying an appropriate cyclic permutation of the sub-braids $T(p_i,q_i)$, we can reorder the triple in either ascending or descending order of the $p_i$. If the result order is ascending, we take the inverse of the braid word, which corresponds to reversing the order of the sub-braids and inverting each crossing, thereby obtaining $p_1'\geq p_2'\geq p_3'$. This transformation does not change the number of link components. The process is analogous to the fact that there is only one combinatorial type of necklace that can be formed from three distinguishable beads, up to rotation and reflection. Finally, apply Lemma~\ref{lem:full_twists-2} to ensure that $p_i'>q_i'\geq0$.
\end{proof}

The following lemmas are analogous to Lemma~\ref{lem:final_step}, ~\ref{lem:swapping_p_and_q} and \ref{lem:forming_r_strands}, and their proofs follow by similar arguments.
\begin{lemma}\label{lem:final_step-2}
    For a $T$-link $T(p_1,q_1;p_2,q_2;p_3,q_3)$ in standard form,
    \[
        NC(p_1,q_1;p_2,q_2;p_3,q_3) = \begin{cases}
            p_1-p_2+NC(p_2,q_2;p_3,q_3)&\text{if}\quad q_1=0\\
            NC(p_1,q_1;p_3,q_3)&\text{if}\quad q_2=0\\
            NC(p_1,q_1;p_2,q_2)&\text{if}\quad q_3=0\\
        \end{cases}
    \]
\end{lemma}
\begin{proof}
    The result follows directly from the definition of the $T$-link.
\end{proof}

\begin{lemma}\label{lem:swapping_p_and_q-2}
    For a $T$-link $T(p_1,q_1;p_2,q_2;p_3,q_3)$ in standard form, if $q_1\geq p_2$, then
    \[
        NC(p_1,q_1;p_2,q_2;p_3,q_3)=NC(q_1,p_1;p_3,-q_3;p_2,-q_2).
    \]
\end{lemma}
\begin{proof}
    Apply an argument analogous to that used in the proof of Lemma~\ref{lem:swapping_p_and_q}, now considering the sub-braid $B=(\sigma_1\cdots\sigma_{p_2-1})^{q_2}(\sigma_1\cdots\sigma_{p_3-1})^{q_3}$.
    When the $(p_1,q_1)$ torus braid is transformed into the $(q_1,p_1)$ torus braid, the vertical order of the two factors of $B$ is reversed. After repositioning the resulting sub-braid and changing all its crossings, which preserves the number of components, we obtain the generalized $T$-link $T(q_1,p_1;p_3,-q_3;p_2,-q_2)$.
    This proves the stated identity.
\end{proof}

\begin{lemma}\label{lem:forming_p_2_strands}
    For a $T$-link $T(p_1,q_1;p_2,q_2;p_3,q_3)$ in standard form, if $p_2>q_1>0$, then
    \[
        NC(p_1,q_1;p_2,q_2;p_3,q_3)=NC(q_1,p_2-p_1;p_2,q_1+q_2;p_3,q_3)
    \]
\end{lemma}
\begin{proof}
    This follows by an argument analogous to the one used in the proof of Lemma~\ref{lem:forming_r_strands}. For example, Figure~\ref{fig:T-link} illustrates the transformation of $T(8,4;6,3;5,3)$ into $T(4,-2;6,7;5,3)$. This transformation is not an isotopy, but it preserves the number of components.
\end{proof} 

\begin{figure}[t!]
    \centering
    \includegraphics[width=0.9\textwidth]{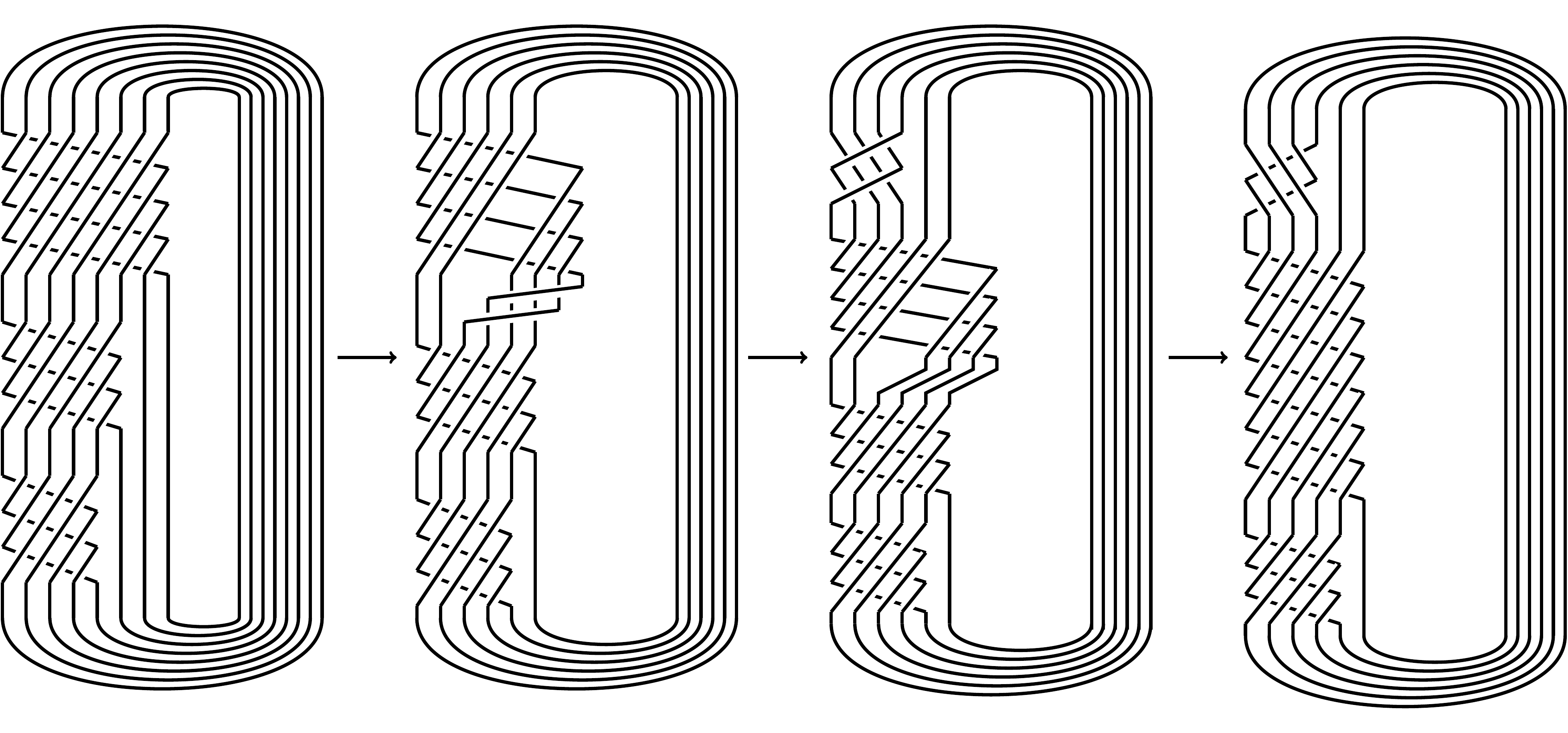}
    \caption{A transformation from $T(8,4;6,3;5,3)$ to $T(4,-2;6,7;5,3)$ that preserves the number of components.}
    \label{fig:T-link}
\end{figure}

We now obtain an algorithm to compute the number of components of a $T$-link with $n=3$, similar to the case of $n=2$.
\subsection*{Algorithm}
Given a (generalized) $T$-link with $n=3$, we first apply Lemma~\ref{lem:standard_form} to convert it into standard form without changing the number of components. If $q_i=0$ for some $i$, then by Lemma~\ref{lem:final_step-2}, the problem reduces to computing the number of components of a twisted torus link. Otherwise, if $q_1\geq p_2$, we apply Lemma~\ref{lem:swapping_p_and_q-2}; or if $p_2>q_1$, we apply Lemma~\ref{lem:forming_p_2_strands} to obtain another generalized $T$-link with the same number of components. After each such transformation, we reapply Lemma~\ref{lem:standard_form} to restore the standard form. This process is repeated until one of the $q_i$ becomes zero. The algorithm terminates because each transformation strictly decreases the maximum number of strands, except when $p_1=p_2$ in Lemma~\ref{lem:forming_p_2_strands}. In that case, the transformed exponent $p_2-p_1$ is zero, so the algorithm terminates at the subsequent reduction step.

Since $\gcd(p_1,q_1,p_2,q_2,p_3,q_3)$ divides the $\gcd$ at each subsequent step of the algorithm, we obtain the following:
\begin{corollary}
    The number of components of a (generalized) $T$-link $T(p_1, q_1; p_2, q_2; p_3,q_3)$ is a positive multiple of $\gcd(p_1,q_1, p_2,q_2, p_3,q_3)$. In particular, if the $T$-link is a knot, then $\gcd(p_1,q_1, p_2,q_2, p_3,q_3)=1$.
\end{corollary}

Our algorithm does not generalize to $T$-links with $n\geq 4$, as the argument used in the proof of Lemma~\ref{lem:standard_form}, which transforms a (generalized) $T$-link into a form where $p_i$ are in descending order, does not extend to higher values of $n$. We conclude by posing the following question.
\begin{question}
    Does our arithmetic reduction procedure extend to (generalized) $T$-links with an arbitrary number of parameter pairs, or is there a closed-form expression for their number of components in terms of the original parameters?
\end{question}

\clearpage
\bibliographystyle{alpha}
\bibliography{references-ttk}

\newcommand{\etalchar}[1]{$^{#1}$}
\begin{thebibliography}{dPHARM26}

\bibitem[BD24]{Bavier-Doleshal-2023}
Brandon Bavier and Brandy Doleshal.
\newblock The {J}ones polynomial for a torus knot with twists.
\newblock {\em Topology Appl.}, 357:Paper No. 109069, 21, 2024.

\bibitem[BDD{\etalchar{+}}17]{LC-2016}
Michelle~S. Berry, Victoria Diaz, Brandy Doleshal, Taylor Martin, Emily~T.
  Winn, and Max Zhou.
\newblock The component number of a twisted torus link.
\newblock {\em Minnesota Journal of Undergraduate Mathematics}, 2017.

\bibitem[BK09]{Birman_09}
Joan Birman and Ilya Kofman.
\newblock A new twist on lorenz links.
\newblock {\em Journal of Topology}, 2(2):227–248, 2009.

\bibitem[BTZ15]{Bridge}
Richard~Sean Bowman, Scott Taylor, and Alexander Zupan.
\newblock Bridge spectra of twisted torus knots.
\newblock {\em Int. Math. Res. Not. IMRN}, (16):7336--7356, 2015.

\bibitem[CDW99]{Callahan-Dean-Weeks-1999}
Patrick~J. Callahan, John~C. Dean, and Jeffrey~R. Weeks.
\newblock The simplest hyperbolic knots.
\newblock {\em J. Knot Theory Ramifications}, 8(3):279--297, 1999.

\bibitem[CKP04]{Champanerkar-Kofman-Patterson-2004}
Abhijit Champanerkar, Ilya Kofman, and Eric Patterson.
\newblock The next simplest hyperbolic knots.
\newblock {\em J. Knot Theory Ramifications}, 13(7):965--987, 2004.

\bibitem[Dea96]{Thesis}
John~Charles Dean.
\newblock {\em Hyperbolic knots with small {S}eifert-fibered {D}ehn surgeries}.
\newblock ProQuest LLC, Ann Arbor, MI, 1996.
\newblock Thesis (Ph.D.)--The University of Texas at Austin.

\bibitem[DHMZ18]{DHMZ-2018}
Brandy~G. Doleshal, Eric~J. Hanson, Taylor~E. Martin, and Max Zhou.
\newblock Twisted torus links and the general braid.
\newblock {\em Minnesota Journal of Undergraduate Mathematics}, 4, 2018.

\bibitem[dP22]{unexpected}
Thiago de~Paiva.
\newblock Unexpected essential surfaces among exteriors of twisted torus knots.
\newblock {\em Algebr. Geom. Topol.}, 22(8):3965--3982, 2022.

\bibitem[dP23]{Paiva-2023-2}
Thiago de~Paiva.
\newblock Hyperbolic twisted torus links.
\newblock {\em Geom. Dedicata}, 217(2):Paper No. 42, 16, 2023.

\bibitem[dP25]{Lorenzknots}
Thiago de~Paiva.
\newblock Satellite knots that cannot be represented by positive braids with
  full twists.
\newblock {\em New York J. Math.}, 31:1690--1701, 2025.

\bibitem[dPHARM26]{Volumebounds}
Thiago de~Paiva, Connie On~Yu Hui, and Jos\'e Andr\'es Rodr\'iguez~Migueles.
\newblock {\bibrev{Generalised {T}-links Represent All Links in the 3-Sphere}}.
\newblock \bibrev{preprint, arXiv:2410.04391}, \bibrev{2026}.

\bibitem[LdP22]{LeeThiago}
Sangyop Lee and Thiago de~Paiva.
\newblock Torus knots obtained by negatively twisting torus knots.
\newblock {\em J. Knot Theory Ramifications}, 31(1):Paper No. 2150080, 19,
  2022.

\bibitem[Lee15]{LeeTorusknotsobtained}
Sangyop Lee.
\newblock Torus knots obtained by twisting torus knots.
\newblock {\em Algebr. Geom. Topol.}, 15(5):2819--2838, 2015.

\bibitem[Lee18]{lee2018satellite}
Sangyop Lee.
\newblock Satellite knots obtained by twisting torus knots: Hyperbolicity of
  twisted torus knots.
\newblock {\em International Mathematics Research Notices}, 2018(3):785--815,
  2018.

\bibitem[MH05]{Morton-Hadji-2005}
Hugh~R. Morton and Richard~J. Hadji.
\newblock Conjugacy for positive permutation braids.
\newblock {\em Fund. Math.}, 188:155--166, 2005.

\bibitem[Mor06]{Morton-2006}
Hugh~R. Morton.
\newblock The {A}lexander polynomial of a torus knot with twists.
\newblock {\em J. Knot Theory Ramifications}, 15(8):1037--1047, 2006.

\bibitem[MS09]{Moriah-Sedgwick-2009}
Yoav Moriah and Eric Sedgwick.
\newblock Heegaard splittings of twisted torus knots.
\newblock {\em Topology and its Applications}, 156(5):885--896, 2009.

\bibitem[PA25]{Adnan-Park-2024}
Kyungbae Park and Adnan.
\newblock The {A}lexander polynomial of twisted torus knots.
\newblock {\em J. Knot Theory Ramifications}, 34(8):Paper No. 2550038, 23,
  2025.

\bibitem[Vaf15]{Vafaee-2015}
Faramarz Vafaee.
\newblock On the knot {F}loer homology of twisted torus knots.
\newblock {\em Int. Math. Res. Not. IMRN}, (15):6516--6537, 2015.

\end{thebibliography}

\end{document}